\documentclass{amsart}
\usepackage{amsmath,amssymb,amsthm}
\usepackage{fullpage}
\usepackage{graphicx}

\newtheorem{lemma}{Lemma}
\newtheorem{theorem}{Theorem}

\usepackage{hyperref}

\usepackage{color}

\usepackage{charter}
\usepackage{euler}
\usepackage[T1]{fontenc}

\newcommand{\R}{\ensuremath{\mathbb{R}}}

\newcommand{\Z}{\ensuremath{\mathbb{Z}}}

\newcommand{\dis}{\mathrm{dis}}
\newcommand{\codis}{\mathrm{codis}}

\newcommand{\h}{\mathrm{H}}
\newcommand{\gh}{\mathrm{GH}}

\title{Tight upper bound on $d_\gh(S^1,S^{2k+1})$: GPT's short proof}

\author{Henry Adams}
\email{henry.adams@ufl.edu}
\thanks{HA was funded by the NSF CAREER Grant DMS 2540172 and by the Simons Foundation's Travel Support for Mathematicians.}

\begin{document}

\begin{abstract}
We describe GPT-5.6 Sol Pro's short proof that $2\cdot d_\gh(S^1,S^{2k+1}) \le \tfrac{2\pi k}{2k+1}$.
This is the odd-dimensional case of a tight upper bound by Harrison and Jeffs on the Gromov--Hausdorff distance between the circle and spheres.
\end{abstract}

\maketitle

\section{Introduction}

Let $k\ge 1$ be an integer.
We give a short proof of the upper bound $2\cdot d_\gh(S^1,S^{2k+1}) \le \tfrac{2\pi k}{2k+1}$.
This is the odd-dimensional case of a tight upper bound by Harrison and Jeffs for the Gromov--Hausdorff distance between the circle and spheres~\cite{harrison2023quantitative}.

The equalities $2\cdot d_\gh(S^1,S^2)=2\cdot d_\gh(S^1,S^3)=\tfrac{2\pi}{3}$ were proved in~\cite{lim2023gromov}, in part building on~\cite{dubins1981equidiscontinuity}.
The lower bounds $2\cdot d_\gh(S^1,S^{2k})\ge\tfrac{2\pi k}{2k+1}$ and $2\cdot d_\gh(S^1,S^{2k+1})\ge\tfrac{2\pi k}{2k+1}$ were proved in~\cite{GH-BU-VR}, in part building on the homotopy types of Vietoris--Rips complexes of the circle~\cite{AA-VRS1,moy2023vietoris}.
Question~8.1 of~\cite{GH-BU-VR} asked whether equality holds in these lower bounds.
Harrison and Jeffs~\cite{harrison2023quantitative} answered this question by proving matching upper bounds, implying $2\cdot d_\gh(S^1,S^{2k+1})=\tfrac{2\pi k}{2k+1}=2\cdot d_\gh(S^1,S^{2k})$.
A second proof appears in~\cite{rodriguez2026some}.

We prove $2\cdot d_\gh(S^1,S^{2k+1})\le\tfrac{2\pi k}{2k+1}$, and hence $2\cdot d_\gh(S^1,S^{2k+1})=\tfrac{2\pi k}{2k+1}$.
We use the symmetric moment curve $g\colon S^1\to S^{2k+1}$ defined by
\[
g(t)=\tfrac{1}{\sqrt{k+1}}\bigl(\cos t,\sin t,\cos 3t,\sin 3t,\ldots,\cos ((2k+1)t),\sin ((2k+1)t)\bigr).
\]
This curve was defined by Barvinok and Novik in~\cite{barvinok2008centrally}; see also~\cite{barvinok2013neighborliness,smilansky1985convex}.

\begin{figure}[hbt]
\centering
\includegraphics[width=0.45\linewidth]{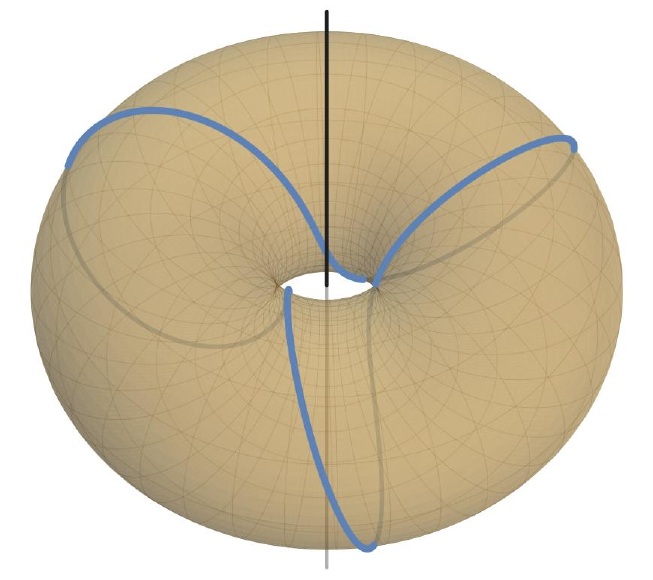}
\caption{The curve $g\colon S^1\to S^3$ given by $g(t)=\frac{1}{\sqrt{2}}(\cos t,\sin t,\cos 3t,\sin 3t)$, which is contained in the torus $S^1 \times S^1 \subseteq S^1*S^1\cong S^3\cong \R^3\cup\{\infty\}$.}
\label{fig:VR}
\end{figure}

The symmetric moment curve $g$ was related to Vietoris--Rips complexes of the circle by Adams, Bush, and Frick in~\cite{ABF,BushThesis}.
Lim, M{\'e}moli, and Smith used variants of this curve to understand $d_\gh(S^1,S^2)$ and $d_\gh(S^1,S^3)$~\cite{lim2023gromov}.
Conversations with Frick and M{\'e}moli about this curve helped initiate the polymath collaboration~\cite{GH-BU-VR}.
As described in~\cite{memoli2024embedding}, I proposed combining the symmetric moment curve with M{\'e}moli and Smith's embedding-projection correspondences to study $d_\gh(S^1,S^{2k+1})$ for all $k$; this proposal is realized here.
A variant of the symmetric moment curve to study $d_\gh(S^1,S^{2k})$ was proposed in~\cite{memoli2024embedding}.
The symmetric moment curve also parametrizes the orbit of a point under the diagonal $S^1$-action on $S^1_1 * S^1_3 * \ldots * S^1_{2k+1}\cong S^{2k+1}$, where $S^1_j$ denotes a copy of $S^1$ with the action $z\cdot w=z^j w$; see~\cite{PersistentEquivariantCohomology}.

What is the value of this new proof?
In addition to being simpler than the proofs in~\cite[Section~5]{harrison2023quantitative} and~\cite[Section~4]{rodriguez2026some}, our symmetric moment curve proof preserves the full symmetry of the circle (avoiding the choice of $n$ evenly-spaced circular points), better preserving the geometry of the problem.
Can these ideas help us better understand the Gromov--Hausdorff distances between spheres, $d_\gh(S^n,S^k)$, for $n<k$~\cite{lim2023gromov,GH-BU-VR}?

\subsection*{AI Disclosure}
The proof was produced by GPT-5.6 Sol Pro in ChatGPT; see\footnote{\url{https://chatgpt.com/share/6a86752b-f5d0-83ea-98d6-f5d1b32039ad}, August 18, 2026} for the prompt and response and see\footnote{\url{https://drive.google.com/file/d/1mzUUsoGNrSkLnov-YpWmdjQXEknZzIFX/view?usp=sharing}} for a PDF of the model's proof.
The author takes responsibility for the mathematics and the writing.

\section{Background and Notation}

The Gromov--Hausdorff distance provides a notion of dissimilarity between metric spaces~\cite{edwards1975structure,gromov1981groups,gromov1981structures,tuzhilin2016invented}.
If $A$ and $B$ are two subsets of a common metric space $Z$, then the \emph{Hausdorff} distance between them is the infimum of all $\varepsilon>0$ such that $A$ is contained in the $\varepsilon$-neighborhood of $B$ in $Z$, and such that $B$ is contained in the $\varepsilon$-neighborhood of $A$ in $Z$:
\[d_\h^Z(A,B)=\inf\{\varepsilon>0~:~A\subseteq B^\varepsilon\text{ and }B\subseteq A^\varepsilon\}.\]
Here, $A^\varepsilon=\{z\in Z~:~d(z,a)<\varepsilon\text{ for some }a\in A\}$, and similarly for $B^\varepsilon$.
For abstract metric spaces $X$ and $Y$, the \emph{Gromov--Hausdorff} distance between them is
\[
d_\gh(X,Y)=\inf_{\substack{\gamma\colon X\hookrightarrow Z \\ \delta\colon Y\hookrightarrow Z}}\{d_\h^Z(\gamma(X),\delta(Y))\},
\]
where the infimum is over all metric spaces $Z$ and isometric embeddings of $X$ and $Y$ into $Z$.
(As written, this is an infimum over a proper class, but to see it is well-defined, note it suffices to consider $Z$ to be the disjoint union of $X$ and $Y$, equipped with a metric extending that on $X$ and $Y$.)

It follows from Kalton and Ostrovskii~\cite{kalton1999distances} that twice the Gromov--Hausdorff distance between metric spaces can be defined as
\begin{equation}
\label{eq:ko}
2\cdot d_\gh(X,Y) = \inf_{f,h}\max\{\dis(f),\dis(h),\codis(f,h)\},
\end{equation}
where $f\colon X\to Y$ and $h\colon Y \to X$ are (possibly discontinuous) functions.
The \emph{distortion} of $f$ is defined as $\dis(f)=\sup_{x,x'\in X}|d(x,x')-d(f(x),f(x'))|$, and the \emph{codistortion} of $f$ and $h$ is defined as $\codis(f,h)=\sup_{x\in X,y\in Y}|d(x,h(y))-d(f(x),y)|$.
%\note{Figure, $S^1$ and $S^2$?}

Equip each sphere with its geodesic metric.
Let $S^1=\R/2\pi\Z$.
For $t,t'\in S^1$, let $d(t,t')=\min_{j\in\Z}|t-t'+2\pi j|$.
For $x,x'\in S^{2k+1}$, we have $d(x,x')=\arccos\langle x,x'\rangle$.

\section{Proof of the theorem}

Define $g\colon S^1\to S^{2k+1}$ by
\[
g(t)=\tfrac{1}{\sqrt{k+1}}\bigl(\cos t,\sin t,\cos 3t,\sin 3t,\ldots,\cos ((2k+1)t),\sin ((2k+1)t)\bigr).
\]
Choose $h\colon S^{2k+1}\to S^1$ to be any function satisfying $d(x,g(h(x)))=d(x,g(S^1))$ for every $x\in S^{2k+1}$; such an $h$ exists because the continuous function $t\mapsto d(x,g(t))$ attains its minimum on the compact space $S^1$.

\begin{theorem}
\label{thm:main}
Let $k\ge 1$ be an integer.
If $x,x'\in S^{2k+1}$ and $t,t'\in S^1$ satisfy $d(x,g(t))=d(x,g(S^1))$ and $d(x',g(t'))=d(x',g(S^1))$, then $\bigl|d(x,x')-d(t,t')\bigr|\leq \tfrac{2\pi k}{2k+1}$.
\end{theorem}

For $t,t'\in S^1$ we deduce $\bigl|d(g(t),g(t'))-d(t,t')\bigr|\leq \tfrac{2\pi k}{2k+1}$; for $x,x'\in S^{2k+1}$ we deduce $\bigl|d(x,x')-d(h(x),h(x'))\bigr|\leq \tfrac{2\pi k}{2k+1}$; and for $x\in S^{2k+1}$ and $t\in S^1$ we deduce $\bigl|d(x,g(t))-d(h(x),t)\bigr|\leq \tfrac{2\pi k}{2k+1}$.
By~\eqref{eq:ko}, we have $2\cdot d_\gh(S^1,S^{2k+1}) \le \max\{\dis(g),\dis(h),\codis(g,h)\}\le \tfrac{2\pi k}{2k+1}$.
Combining with the lower bound~\cite[Remark~5.5]{GH-BU-VR} gives $2\cdot d_\gh(S^1,S^{2k+1})=\tfrac{2\pi k}{2k+1}$.

\begin{proof}[Proof of Theorem~\ref{thm:main}]
For a unit vector $y\in S^{2k+1}$, define $F_y\colon S^1\to \R$ by $F_y(s)=\langle y,g(s)\rangle$.
Note $d(y,g(s))=d(y,g(S^1))$ if and only if $F_y(s)=\max_u F_y(u)$.
Since $g(u+\pi)=-g(u)$, we have $F_y(u+\pi)=-F_y(u)$.
Thus, if $M_y=\max_u F_y(u)$, then $\min_u F_y(u)=-M_y$, so $|F_y(u)|\leq M_y$ for every $u$.
Moreover, writing $y=(a_0,b_0,\ldots,a_k,b_k)$ and using orthogonality of the sine and cosine functions,
%\[\tfrac{1}{2\pi}\int_0^{2\pi}F_y(u)^2\,du=\tfrac{1}{2k+2}\sum_{j=0}^k(a_j^2+b_j^2)=\tfrac{1}{2k+2}\qquad \text{implying} \qquad M_y\geq\tfrac{1}{\sqrt{2k+2}}.\]
%\[
\[
\tfrac{1}{2\pi}\int_{0}^{2\pi} F_y(u)^2\,du
=
\tfrac{1}{2\pi(k+1)}
\sum_{j=0}^{k}
\int_{0}^{2\pi}
(a_j^2\cos^2((2j+1)u)+
b_j^2\sin^2((2j+1)u))\,du 
=
\tfrac{1}{2k+2}\sum_{j=0}^{k}(a_j^2+b_j^2)
=
\tfrac{1}{2k+2}.
\]
Hence $M_y\geq\tfrac{1}{\sqrt{2k+2}}$.
%If $s$ is a maximizing parameter for $F_y$, applying Lemma~\ref{lem:vdcs} to the trigonometric polynomial $u\mapsto \tfrac{F_y(s+u)}{M_y}$ gives $F_y(s+u)\geq M_y\cos((2k+1)|u|)$ for $|u|\leq\tfrac{\pi}{2k+1}$.

Set $\delta_{2k+1}=d(x,x')$ and $\delta_{1}=d(t,t')$.
We first prove
$\delta_{2k+1}\geq\delta_{1}-\tfrac{2\pi k}{2k+1}$.
This is automatic when $\delta_{1}\leq \tfrac{2\pi k}{2k+1}$, so suppose $\delta_{1}>\tfrac{2\pi k}{2k+1}$ and write $\delta_{1}=\tfrac{2\pi k}{2k+1}+\beta$ with $0<\beta\leq\tfrac{\pi}{2k+1}$.
The block-diagonal orthogonal action that rotates the block of frequency $2j+1$ through angle $(2j+1)\theta$ satisfies $g(u+\theta)=R_\theta g(u)$.
Applying $R_{-t}$ simultaneously to $x$ and $x'$, and then applying the reflection taking $g(u)$ to $g(-u)$ if necessary, preserves all relevant distances and maximizing conditions.
We may therefore assume
\[t=0, \qquad t'=\delta_{1}=\pi-\eta, \qquad \eta=\tfrac{\pi}{2k+1}-\beta.\]
Let $u_0=\pi-\tfrac{\eta}{2}$ be the midpoint, in the chosen coordinate interval, between $t'$ and $\pi$.
Since $0$ maximizes $F_x$, applying Lemma~\ref{lem:vdcs} to the trigonometric polynomial $u\mapsto \tfrac{F_x(u)}{M_x}$ gives $F_x(u)\geq M_x\cos((2k+1)|u|)$ for $|u|\leq\tfrac{\pi}{2k+1}$.
Hence $F_x(u_0)=-F_x(\tfrac{-\eta}{2})\leq -M_x\cos\tfrac{(2k+1)\eta}{2}$.
Since $t'$ maximizes $F_{x'}$ and $u_0-t'=\tfrac{\eta}{2}$, applying Lemma~\ref{lem:vdcs} to $u\mapsto\tfrac{F_{x'}(u+t')}{M_{x'}}$ gives $F_{x'}(u_0)\geq M_{x'}\cos\tfrac{(2k+1)\eta}{2}$.
Because $(2k+1)\eta=\pi-(2k+1)\beta$, we have $\cos\tfrac{(2k+1)\eta}{2}=\sin\tfrac{(2k+1)\beta}{2}$.
It follows that
\[\|x'-x\| \geq \langle x'-x,g(u_0)\rangle = F_{x'}(u_0)-F_x(u_0) \geq (M_x+M_{x'})\sin\tfrac{(2k+1)\beta}{2} \geq \tfrac{2}{\sqrt{2k+2}}\sin\tfrac{(2k+1)\beta}{2} \geq 2\sin\tfrac{\beta}{2},\]
where we used $M_x,M_{x'}\geq\tfrac{1}{\sqrt{2k+2}}$ and Lemma~\ref{lem:sine} with $z=\tfrac{\beta}{2}$.
Since $\|x'-x\|=2\sin\tfrac{\delta_{2k+1}}{2}$ and since sine is increasing on $[0,\frac{\pi}{2}]$, we obtain $\delta_{2k+1}\geq\beta=\delta_{1}-\tfrac{2\pi k}{2k+1}$.

We now prove $\delta_{2k+1}\leq\delta_{1}+\tfrac{2\pi k}{2k+1}$.
If $\delta_{1}\geq\tfrac{\pi}{2k+1}$, then $\delta_{1}+\tfrac{2\pi k}{2k+1}\geq\pi$, so this bound is automatic.
Suppose $\delta_{1}<\tfrac{\pi}{2k+1}$.
Since $d(x',g(t'))=d(x',g(S^1))$, we have $d(-x',g(t'+\pi))=d(-x',g(S^1))$.
Moreover, $d(t,t'+\pi)=\pi-\delta_{1}>\tfrac{2\pi k}{2k+1}$.
Applying $\delta_{2k+1}\geq\delta_{1}-\tfrac{2\pi k}{2k+1}$ to $x$ and $-x'$, with $d(x,g(t))=d(x,g(S^1))$ and $d(-x',g(t'+\pi))=d(-x',g(S^1))$, gives
$\pi-\delta_{2k+1}=d(x,-x')\geq d(t,t'+\pi)-\tfrac{2\pi k}{2k+1} = (\pi-\delta_{1})-\tfrac{2\pi k}{2k+1}=\tfrac{\pi}{2k+1}-\delta_{1}$.
Hence $\delta_{2k+1}\leq \tfrac{2\pi k}{2k+1}+\delta_{1}$.
\end{proof}

\section{Proof of the lemmas}

The proof above used two lemmas, stated here.
The first lemma follows from a variant of Bernstein's inequality for trigonometric polynomials, which goes back nearly a century to Szeg{\H{o}} in 1928~\cite{Szego1928}.
See also van der Corput and Schaake in 1935~\cite{VanDerCorputSchaake1935}, and a modern reference by Queff{\'e}lec and Zarouf in 2019~\cite{QueffelecZarouf2019}.

\begin{lemma}[p.~69 of~\cite{Szego1928}, Satz~8 of~\cite{VanDerCorputSchaake1935}, Theorem~3.7 of~\cite{QueffelecZarouf2019}]
\label{lem:vdcs}
Let $m\ge 1$, and let $T$ be a real trigonometric polynomial of degree at most $m$ such that $|T(u)|\leq 1$ for every $u$.
Then $|T'(u)|\leq m\sqrt{1-T(u)^2}$ for $u\in\R$.
If furthermore $T(0)=1$, then $T(u)\geq \cos(m|u|)$ for $|u|\leq \tfrac{\pi}{m}$.
\end{lemma}

\begin{proof}
Theorem~3.7 of~\cite{QueffelecZarouf2019} proves $(T'(u))^2+m^2(T(u))^2\le m^2$, yielding $|T'(u)|\leq m\sqrt{1-T(u)^2}$.

Now assume $T(0)=1$.
For $0<\varepsilon<1$, put $T_\varepsilon=(1-\varepsilon)T$ and $\theta_\varepsilon(u)=\arccos T_\varepsilon(u)$.
The first part of the lemma applied to $T_\varepsilon$ gives
$|T_\varepsilon'(u)|\le m\sqrt{1-T_\varepsilon(u)^2}$.
Since $|T_\varepsilon|<1$, the function $\theta_\varepsilon$ is differentiable.
Hence $|\theta_\varepsilon'(u)|=\tfrac{|T_\varepsilon'(u)|}{\sqrt{1-T_\varepsilon(u)^2}}\leq m$.
Therefore $\theta_\varepsilon(u) \leq \theta_\varepsilon(0)+m|u|$.
Letting $\varepsilon\to0^+$ yields $\arccos T(u)\leq m|u|$.
If $|u|\leq\tfrac{\pi}{m}$, both sides lie in $[0,\pi]$, on which cosine is decreasing.
It follows that $T(u)\geq \cos(m|u|)$ for $|u|\leq \tfrac{\pi}{m}$.
\end{proof}

The second lemma is elementary.

\begin{lemma}\label{lem:sine}
For every odd integer $m\geq3$ and every $0\leq z\leq\tfrac{\pi}{2m}$, we have $\tfrac{\sin(mz)}{\sqrt{m+1}}\geq\sin z$.
\end{lemma}

\begin{proof}
The claim is clear for $z=0$; let $z>0$.
For $m=3$, we have $\tfrac{\sin(3z)}{\sin z}=3-4\sin^2z\geq2=\sqrt{4}$ because $0< z\leq\frac{\pi}{6}$.
Now let $m\geq5$.
Since $0\leq mz\leq\frac{\pi}{2}$, concavity of sine on $[0,\frac{\pi}{2}]$ gives $\sin(mz)\geq\tfrac{2mz}{\pi}$, while $\sin z\leq z$.
Also $\tfrac{2m}{\pi}\geq\sqrt{m+1}$ for $m\geq5$; for example, $\pi^2<10$ and $4m^2-10(m+1)>0$ for every $m\geq5$.
Hence $\sin(mz)\geq\tfrac{2mz}{\pi}\geq \sqrt{m+1}\,z\geq\sqrt{m+1}\sin z$.
\end{proof}

\bibliographystyle{plain}
\bibliography{TightUpperBound-GPTsShortProof}

\end{document}